\documentclass[12pt]{amsart}
\usepackage{amssymb}%
\usepackage[normalem]{ulem}
\usepackage{amsthm}%
\usepackage{amsmath}%
\usepackage{amsfonts} %This is to have nice letters
\usepackage{epsfig}%
\usepackage{epic}%
\usepackage{amscd}%
\usepackage{color}%
\usepackage[linktoc=all, pagebackref, hyperindex]{hyperref}%
\hypersetup{colorlinks,  citecolor=blue,  filecolor=blue,  linkcolor=red,  urlcolor=black}
\theoremstyle{plain}
  \newtheorem{thm}{Theorem}[section]
  \newtheorem{prop}[thm]{Proposition}
  \newtheorem{lem}[thm]{Lemma}

\theoremstyle{definition}

  \newtheorem{qu}[thm]{Question}
\theoremstyle{remark}
  \newtheorem{rem}[thm]{Remark}

    {\begin{list}
        {\noindent\makebox[0mm][r]{\rm(\arabic{enumi})}}
        {\leftmargin=6.3ex \usecounter{enumi}}
    }
    {\end{list}}
\newcommand{\Mon}{\operatorname{Mon}}

\newcommand{\supp}{\operatorname{supp}}

\newcommand{\Proj}{\operatorname{Proj}}
\newcommand{\Log}{\operatorname{Log}}

\def\PP{{\mathbb P}}

\def\GG{{\mathbb G}}
\def\KK{{\mathbb K}}
\def\1{{\mathbf 1}}
\def\a{{\mathbf a}}

\def\gg{{\mathbf g}}

\def\<{{\langle}}
\def\>{{\rangle}}

\newcommand{\excise}[1]{}
\begin{document}

\title{The Gauss Algebra of squarefree Veronese algebras}

\author{Somayeh Bandari}
\address{Department of Mathematics, Buein Zahra Technical University, Buein Zahra, Qazvin, Iran}
\email{somayeh.bandari@yahoo.com and s.bandari@bzte.ac.ir}

\author{Raheleh Jafari}
\address{Mosaheb Institute of Mathematics, Kharazmi University,
    Tehran, Iran}
\email{rjafari@khu.ac.ir}

\makeatletter
\@namedef{subjclassname@2020}{\textup{2020} Mathematics Subject Classification}
\makeatother



\maketitle

\begin{abstract}
 We investigate the  Gauss algebra for squarefree Veronese
algebras generated in degree $3$. For small dimensions not exceeding
$7$, we  determine the Gauss algebra by specifying its  generators
and show in particular that it is normal and Cohen-Macaulay.
\end{abstract}

\section{Introduction}

Let $S=\KK[x_1,\dots,x_d]$ be a polynomial ring over $\KK$, where
$\KK$ is a field of characteristic zero,  and
    $\gg=g_1,\dots,g_n$
be a sequence of non-constant homogeneous polynomials of the same
degree in $S$ generating the $\KK$-subalgebra $A=\KK[\gg]\subseteq
S$ of dimension $d$. Then the  {\em Jacobian matrix} of $\gg$,
denoted by $\Theta(\gg)$, is a $n\times d$ matrix of rank $d$,
\cite[Proposition~1.1]{Simis-2003}. In this situation, the
$\KK$-subalgebra generated by the set of $d\times d$ minors of
$\Theta(\gg)$ is called the {\em Gauss algebra} associated with
$\gg$, \cite[Definition~2.1]{BGS}. Note that  the definition does
not depend on the choice of the homogeneous generators of $A$, so we
simply denote the Gauss algebra associated with $\gg$, by $\GG(A)$,
and call it the Gauss algebra of $A$. The Gauss algebra $\GG(A)$ is
isomorphic to the homogeneous coordinate ring of the image of the
Gauss map associated to the rational map
$\PP^{d-1}\longrightarrow\PP^{n-1}$, given by
$\a=(a_1:\cdots:a_d)\mapsto (g_1(\a):\cdots:g_n(\a))$. Moreover,
there exists an injective homomorphism of $\KK$-algebras
$\GG(A)\hookrightarrow  A$ inducing the rational map from $\Proj
(A)$ to its Gauss image, \cite[Corollary~2.4]{BGS}. In light of
these facts, $\GG(A)$ warrants study as an  object of critical
importance.  P.~Brumatti, P.~Gimenez and A.~Simis, made a step
toward this goal, by clarifying the structure of the Gauss algebra
for some known toric algebras  in \cite{BGS}.  If $A$ is a toric
algebra, generated by monomials $\gg = g_1,\dots, g_n$ of degree
$r$, then all minors of $\Theta(\gg)$ are monomials of degree
$(r-1)d$ in $S$. In particular, the Gauss algebra is again a toric
algebra generated by monomials of degree $(r-1)d$.  The second
author, in a joint study with J.~Herzog and A.~Nasrollah Nejad in
\cite{HJN}, described   the generators and the structure of
$\GG(A)$,  when $A$ is a Borel fixed algebra, a squarefree Veronese
algebra generated in degree 2, or the edge ring of a bipartite graph
with at least one loop.

The Gauss algebra associated with a Veronese algebra is again a
Veronese algebra, see \cite[Proposition~3.2]{BGS} or
\cite[Corollary~2.3]{HJN}. The situation for squarefree Veronese
algebra is more complicated.  In \cite{HJN},  a full description of
$\GG(A)$, when $A=\KK[V_{2,d}]$ is a  squarefree Veronese algebra
generated in degree $2$ is given.   Note that for $d\leq3$, the
Gauss algebra is isomorphic to a polynomial ring. Let $\Mon_S(t,r)$
denote the set of all monomials $u$ of degree $r$ in $S$, such that
$|\supp(u)|\geq t$, where $\supp(u)=\{i : x_i |u\}$. The monomial
ideal generated by $\Mon_S(t,r)$ is polymatroidal. In particular,
the $\KK$-algebra $\KK[\Mon_S(t,r)]$ is normal and Cohen-Macaulay,
\cite[Proposition~3.1]{HJN}.

It is shown  in \cite{HJN} that
$\GG(\KK[V_{2,4}])=\KK[\Mon_S(3,4)\setminus\{x_1x_2x_3x_4\}]$ and
for $d\geq 5$, $\GG(\KK[V_{2,d}])=\KK[\Mon_S(3,d)]$. Algebras of
this type may be viewed as the base ring of a polymatroid.  In
particular,  $\GG(\KK[V_{2,d}])$ is a normal and Cohen-Macaulay
domain if $d\geq 5$. However, $\GG(A)$ is not normal for $d=4$.

In this paper, we investigate the  Gauss algebra for squarefree
Veronese algebras generated in degree $3$, describing the generators
in case of $d\leq 7$.  Note that for $d=3,4$,  $\GG(\KK[V_{3,d}])$
is isomorphic to a one dimensional polynomial ring.

For $d=5,6,7$, we show that
$\GG(\KK[V_{3,d}])=\KK[\Mon^*_S(4,2d)\setminus E_d]$, where
$$\Mon^*_S(4,2d)= \{x_1^{a_1}\cdots x_d^{a_d}\in \Mon_S(4,2d)\; ; \; a_i\leq d-2 \;
\text{for}\; 1\leq i\leq
    d\},$$
and
$$E_d=\{u\in \Mon^*_S(4,2d) \ ; \  u=x_{i_1}^{a_1}x_{i_2}^{a_2}x_{i_3}^{a_3}x_{i_4}
\text{ for some } i_1,i_2,i_3,i_4\in[d]
\}.$$

The monomial ideal generated by
$\Mon^*_S(4,2d)\setminus E_d$ is polymatroidal for all $d\geq 5$,
see Theorem~\ref{poly}.  In particular, $\GG(\KK[V_{3,d}])$ is
normal and Cohen-Macaulay. Based on our evidences, we guess
$\KK[\Mon^*_S(4,2d)\setminus E_d$] must define the Gauss algebra
$\GG(\KK[V_{3,d}])$  for all $d\geq5$. We leave the case of higher
dimensions $d\geq 8$, as an open question.

\section{The Gauss algebra of squarefree Veronese algebras}

Let $S=\KK[x_1,\ldots,x_d]$ be a polynomial ring  over $\KK$, where
$\KK$ is a field of characteristic  zero. Let $A=\KK[\gg]$ be a
toric algebra, generated by a sequence of monomials
$\gg=g_1,\ldots,g_n$, where $g_i= {x_1^{a_{1i}}\cdots x_d^{a_{di}}}$
for $i=1,\ldots,n$. We denote the $r$-minor corresponding to the
rows $i_1,\dots,i_r$ and columns $j_1,\dots,j_r$ of a matrix $B$, by
$[i_1,\dots,i_r\ | \ j_1,\ldots,j_r ]_B$. Then
\[x_{j_1}\ldots x_{j_r} [i_1,\ldots,i_r\ | \ j_1,\ldots,j_r ]_{\Theta(\gg)}=g_{i_1}\ldots g_{i_r}[i_1,\ldots,i_r\ | \ j_1,\ldots,j_r ]_{\mathrm{Log}(\gg)},\]
where $\Theta(\gg)$ is the Jacobian matrix of $\gg$ and
 $\mathrm{Log}(\gg)=(a_{ij})$ is the exponent matrix of $\gg$, whose
columns  are the exponents of each monomial in $\gg$. It follows
that the Gauss algebra $\GG(A)$ can be described in terms of
products of $d$ monomials in $\gg$,  as the following  result shows.

\begin{prop}\label{Gauss} \cite[Page 4]{HJN}
Let $A=\KK[\gg]$ be a homogeneous monomial algebra in $S$  with
generators $\gg=g_1,\ldots,g_n$. Then the Gauss algebra of $A$ is
determined as
$$\GG(A)=\KK[\frac{g_{i_1}\cdots g_{i_d}}{x_1\cdots x_d} \ ; \
\det\Log(g_{i_1},\ldots,g_{i_d})\neq 0].$$
\end{prop}

Let $V_{r,d}$ denote  the set of all squarefree monomials of degree
$r$ in $S$. Then the $\KK$-algebra $A=\KK[V_{r,d}]$ is called the
{\em squarefree Veronese algebra} of degree $r$ in $S$.

\begin{rem}

Note that $V_{r,d}$ is a non-empty set  if and only if $r\leq d$.
For $r=d$, it has only one element   $g=x_1\cdots x_d$ which means
that $A=\KK[V_{r,d}]$ is a one dimensional polynomial ring. In the
case that $r=d-1$, $V_{r,d}$  has $d$ elements
$\{g_i=x_1\cdots\hat{x_i}\cdots x_d \;;\; 1\leq i\leq d\}$ and
consequently, $\GG(A)=\KK[\frac{g_1\cdots g_d}{x_1\cdots x_d}]$.
 Therefore,  the Gauss algebra is isomorphic to a one dimensional polynomial
 ring. 
\end{rem}

Denote by $\Mon_S(t,r)$ the set of all monomials $u$ of degree $r$
in $S$ such that $|\supp(u)|\geq t$, where $\supp(u) = \{i ;  x_i |
u\}$.
 By  \cite[Remark 3.4]{HJN}, the Gauss algebra of $A=\KK[V_{r,d}]$ with $d\geq r+2$, is contained in
 $\KK[\Mon_S(r+1,(r-1)d)]$. Indeed,
\begin{equation}\label{GMon}
    \mathbb{G}(A)\subseteq \KK[\Mon_S^*(r+1,(r-1)d)],
\end{equation}
where
$$\Mon_S^*(r+1,(r-1)d)=\{x_1^{a_1}\cdots x_d^{a_d}\in \Mon_S(r
+ 1,(r-1)d)\;;\; a_i\leq d-2 \; \text{for}\; 1\leq i\leq
    d\}.$$
We  define the set
$$E_d=\{u\in \Mon^*_S(4,2d) \ ; \  u=x_{i_1}^{a_1}x_{i_2}^{a_2}x_{i_3}^{a_3}x_{i_4}
    \text{ for some } i_1,i_2,i_3,i_4\in[d]
    \}.$$

\begin{lem}\label{Ed}
 Let  $d\geq 5$. Then
$\GG(\KK[V_{3,d}])\subseteq \KK[\Mon_S^*(4,2d)\setminus E_d]$.
\end{lem}
\begin{proof}
 By (\ref{GMon}), we have  $\GG(\KK[V_{3,d}])\subseteq
\KK[\Mon_S^*(4,2d)]$. Now assume, for the sake of contradiction,
that there exists a monomial
$m=x_{i_1}^{a_1}x_{i_2}^{a_2}x_{i_3}^{a_3}x_{i_4}\in
\GG(\KK[V_{3,d}])\cap E_d$. Then, there exist monomials
$g_1,\dots,g_d\in V_{3,d}$ such that $m=\frac{g_1\dots g_d}{x_1\dots
x_d}$ with $\det\Log(g_1,\dots,g_d)\neq0$. Let us denote  the
columns of $\Log(g_1,\dots,g_d)$ by $C_1,\dots,C_d$ and its rows by
$R_1,\dots,R_d$. After suitable relabeling, we may assume that,
$R_4$ has two non-zero entries and each of $R_5,\dots,R_d$ has one
non-zero entry.

As the rows are different, there are $d-4$ columns, say $C_5,\dots,
C_{d}$ with a non-zero entry at $j$th row, for some $5\leq j \leq
d$, and the rest of $4$ columns, $C_1, \dots, C_4$ have zero entries
on the rows $R_5,\dots,R_d$. Note that $R_4$ has only two non-zero
entries,  so there are  two columns among $C_1,\dots,C_4$, with tree
non-zero elements in the first three rows, which means having  two identical columns, a contradiction.
\end{proof}

\begin{lem}\label{som}
Let $d\in\{5,6,7\}$. Then
    $E_d=\{x_{i_1}^{d-2}x_{i_2}^{a_2}x_{i_3}^{a_3}x_{i_4}
     \ ; \  1\leq a_2,a_3\leq d-2\;,\; a_2+a_3=d+1  \  \text{ and } i_j\in[d]\}$.
\end{lem}
\begin{proof}
Let $m=x_{i_1}^{a_1}x_{i_2}^{a_2}x_{i_3}^{a_3}x_{i_4}\in E_d$. Then
$a_i\leq d-2$ for $i=1,2,3$. If $a_i\leq d-3$ for $i=1,2,3$, then
     $2d-1=a_1+a_2+a_3\leq 3(d-3)$, which implies $8\leq d$, a contradiction.
\end{proof}

\begin{lem}\label{supp<d}
 Let $d\geq 6$ and
$\GG(\KK[V_{3,d-1}])=\KK[\Mon_{S'}^*(4,2(d-1))\setminus E_{d-1}]$,
where $S'=\KK[x_{i_1},\ldots,x_{i_{d-1}}]\subset S$. If $m$ is a
monomial in $\KK[\Mon^*_S(4,2d)\setminus E_d]$ with $|\supp(m)|<d$,
then $m\in\GG(\KK[V_{3,d}])$.
\end{lem}
\begin{proof}
Let $|\supp(m)|=t<d$. Without loss of generality, we may assume that
$m=x_1^{a_1}\cdots x_t^{a_t}$. Since $a_1+\cdots+a_t=2d$ and
$a_i\leq d-2$ for all $i\in[t]$,  there exist $1\leq r<s\leq t$ with
$a_r,a_s>1$.  We choose $a_r,a_s$ among the largest  exponents. Let
$u=\frac{m}{x_rx_s}=x_1^{a'_1}\cdots x_t^{a'_t}\in
\Mon_{S'}(4,2(d-1))$, where $S'=\KK[x_1,\ldots,x_{d-1}]$.  Note that
since at most two of powers $a_i$ are equal $d-2$, it follows that
$u\in \Mon^*_{S'}(4,2(d-1))$. If $u\in E_{d-1}$, then $t=4$ and
$2\in\{a_r, a_s\}$. Since $m\not\in E_d$, we may assume that
$a_1=a_r, a_2=a_3=a_4=2$. Note that $a_1\leq d-2$, which means that
$2d=a_1+6\leq d+4$, consequently $d\leq 4$, a contradiction.
Therefore, $u\notin E_{d-1}$. Hence, by assumption, $u\in
\mathbb{G}(\KK[V_{3,d-1}])$. Then $u=\frac{g_1\cdots
g_{d-1}}{x_1\cdots
    x_{d-1}}$ such that  $g_i\in V_{3,d-1}$ and
$\det(\Log(g_1,\ldots,g_{d-1}))\neq 0$. Let $g_d=x_rx_sx_d$. So
$(x_1\cdots x_d)m=x_1\cdots x_dx_rx_su=g_1\cdots g_d$. Since all the
entries of the last row of $\Log(g_1,\ldots,g_d)$ are zero, except
the last one, which is equal to $1$, we see that
$\Log(g_1,\ldots,g_d)$ is non-singular.
\end{proof}

\begin{prop}\label{degree5}
Let $A=\KK[V_{3,5}]$ and $S=\KK[x_1,\ldots,x_5]$. Then
$$\mathbb{G}(A)=\KK[\Mon_S^*(4,10)\setminus\{x_{i_1}^3x_{i_2}^3x_{i_3}^3x_{i_4}\; ; \;
i_j\in[5]\}].$$
\end{prop}

\begin{proof}
Note that  by Lemma \ref{som},
$E_5=\{x_{i_1}^3x_{i_2}^3x_{i_3}^3x_{i_4}\; ; \; i_j\in[5]\}$, and by
Lemma~\ref{Ed}, $\mathbb{G}(A)\subseteq\KK[\Mon_S^*(4,10)\setminus
E_5]$.
Now, it remains to show that  any monomial
$m\in\Mon_S^*(4,10)\setminus E_5$ can be written as $\frac{g_1\cdots
g_5}{x_1\cdots x_5}$ with   $g_i\in V_{3,5}$ and
$\det(\Log(g_1,\ldots,g_5))\neq 0$. We distinguish the following
cases, after a suitable relabeling:

\textbf{Case 1:} For $m=x_1^2x_2^2x_3^2x_4^2x_5^2$, let
$g_1=x_1x_2x_3, g_2=x_2x_3x_4, g_3=x_3x_4x_5, g_4=x_1x_4x_5$ and
$g_5=x_1x_2x_5$. Then, the exponent vectors $\a_1=(1,1,1,0,0),
\a_2=(0,1,1,1,0), \a_3=(0,0,1,1,1), \a_4=(1,0,0,1,1)$ and
$\a_5=(1,1,0,0,1)$ are linearly independent and we are done.

\textbf{Case 2:} If $m=x_1^3x_2^2x_3^2x_4^2x_5$, then
$g_1=x_1x_2x_3, g_2=x_1x_2x_4, g_3=x_1x_2x_5, g_4=x_1x_3x_4$ and
$g_5=x_3x_4x_5$ will work.

\textbf{Case 3:} If $m=x_1^3x_2^3x_3^2x_4x_5$, then $g_1=x_1x_2x_3, g_2=x_1x_2x_4, g_3=x_1x_2x_5,
g_4=x_1x_3x_4$ and $g_5=x_2x_3x_5$ will work.
\end{proof}

\textbf{Case 4:} $|\supp(m)|=4$. Without loss of generality, we
may assume that $m=x_1^{a_1}\cdots x_4^{a_4}$. Since
$a_1+a_2+a_3+a_4=10$ and $a_i\leq 3$ for all $i\in[4]$,  there exist
 $1\leq r<s\leq 4$ with $a_r,a_s>1$.  We choose
$a_r,a_s$ among the largest  exponents. Let
$u=\frac{m}{x_rx_s}=x_1^{a'_1}\cdots x_4^{a'_4}\in \Mon_{S'}(4,8)$,
where $S'=\KK[x_1,\ldots,x_4]$ and $a'_i\leq 2$. Since
$a'_1+\cdots+a'_4=8$ and $a'_i\leq 2$ for all $i\in[4]$, it follows
that  $a'_i=2$  for all $i\in[4]$. Let $g_1=x_1x_2x_3,
g_2=x_2x_3x_4, g_3=x_1x_3x_4$ and $g_4=x_1x_2x_4$. So
$u=x_1^2x_2^2x_3^2x_4^2=\frac{g_1\cdots g_4}{x_1\cdots x_4}$ and
$\det(\Log(g_1,\ldots,g_4))\neq 0$. Now, let $g_5=x_rx_sx_5$. So
$(x_1\cdots x_5)m=x_1\cdots x_5x_rx_su=g_1\cdots g_5$. Since all the
entries of the last row of $\Log(g_1,\ldots,g_5)$ are zero, except
the last one, which is equal to $1$, we see that
$\Log(g_1,\ldots,g_5)$ is non-singular.

\begin{prop}\label{degree6}
Let $A=\KK[V_{3,6}]$  and $S=\KK[x_1,\ldots,x_6]$.  Then
$$\mathbb{G}(A)=\KK[\Mon_S^*(4,12)\setminus\{x_{i_1}^4x_{i_2}^4x_{i_3}^3x_{i_4}\; ; \;
i_j\in[6]\}].$$
\end{prop}

\begin{proof}
Note that  by Lemma \ref{som},
$E_6=\{x_{i_1}^4x_{i_2}^4x_{i_3}^3x_{i_4}\; ; \;
    i_j\in[6]\}$, and by Lemma~\ref{Ed}, $\mathbb{G}(A)\subseteq\KK[\Mon_S^*(4,12)\setminus E_6]$.
Now, it remains to show that any
monomial
 $m\in\Mon_S^*(4,12)\setminus E_6$ can be written as $\frac{g_1\cdots g_6}{x_1\cdots x_6}$
with  $g_i\in V_{3,6}$ and $\det(\Log(g_1,\ldots,g_6))\neq 0$. The
case that $|\supp(m)|<6$,  follows by Lemma~\ref{supp<d}  and
Proposition \ref{degree5}. If $|\supp(m)|=6$, we distinguish the
following cases, after a suitable relabeling:

\textbf{Case 1:} For $m=x_1^2\cdots x_6^2$, let  $g_1=x_1x_2x_3,
g_2=x_1x_4x_5, g_3=x_1x_3x_6, g_4=x_2x_4x_6, g_5=x_2x_5x_6$ and
$g_6=x_3x_4x_5$. Then, the exponent vectors $\a_1=(1,1,1,0,0,0),
\a_2=(1,0,0,1,1,0), \a_3=(1,0,1,0,0,1), \a_4=(0,1,0,1,0,1),
\a_5=(0,1,0,0,1,1)$ and $\a_6=(0,0,1,1,1,0)$ are linearly
independent and we are done.

\textbf{Case 2:} If $m=x_1^3x_2^2x_3^2x_4^2x_5^2x_6$, then
$g_1=x_1x_2x_3, g_2=x_1x_3x_4, g_3=x_1x_4x_6, g_4=x_1x_5x_6,
g_5=x_2x_3x_5$ and $g_6=x_2x_4x_5$ will work.

\textbf{Case 3:} If $m=x_1^3x_2^3x_3^2x_4^2x_5x_6$, then
$g_1=x_1x_2x_4, g_2=x_1x_2x_5, g_3=x_1x_2x_6, g_4=x_1x_3x_4,
g_5=x_2x_3x_5$ and $g_6=x_3x_4x_6$ will work.

\textbf{Case 4:} If $m=x_1^3x_2^3x_3^3x_4x_5x_6$, then
$g_1=x_1x_2x_3, g_2=x_1x_2x_4, g_3=x_1x_2x_5, g_4=x_1x_3x_6,
g_5=x_2x_3x_6$ and $g_6=x_3x_4x_5$ will work.

\textbf{Case 5:} If $m=x_1^4x_2^2x_3^2x_4^2x_5x_6$,
then$g_1=x_1x_2x_3, g_2=x_1x_2x_4, g_3=x_1x_2x_5, g_4=x_1x_3x_4,
g_5=x_1x_4x_6$ and $g_6=x_3x_5x_6$ will work.

\textbf{Case 6:} If $m=x_1^4x_2^3x_3^2x_4x_5x_6$, then
$g_1=x_1x_2x_3, g_2=x_1x_2x_4, g_3=x_1x_2x_5, g_4=x_1x_2x_6,
g_5=x_1x_3x_4$ and $g_6=x_3x_5x_6$ will work.

\textbf{Case 7:} If $m=x_1^4x_2^4x_3x_4x_5x_6$, then $g_1=x_1x_2x_3,
g_2=x_1x_2x_4, g_3=x_1x_2x_5, g_4=x_1x_2x_6, g_5=x_1x_3x_4$ and
$g_6=x_2x_5x_6$ will work.
\end{proof}

\begin{prop}\label{degree7}
Let $A=\KK[V_{3,7}]$  and $S=\KK[x_1,\ldots,x_7]$.  Then
$$\mathbb{G}(A)=\KK[\Mon_S^*(4,14)\setminus\{x_{i_1}^5x_{i_2}^5x_{i_3}^3x_{i_4}\;,\, x_{i_1}^5x_{i_2}^4x_{i_3}^4x_{i_4}\; ;\;
i_j\in[7]\}].$$
\end{prop}

\begin{proof}
Note that  by Lemma \ref{som},
$E_7=\{x_{i_1}^5x_{i_2}^5x_{i_3}^3x_{i_4}\;,\,
x_{i_1}^5x_{i_2}^4x_{i_3}^4x_{i_4}\;;\; i_j\in[7]\}$, and by
Lemma~\ref{Ed}, $\mathbb{G}(A)\subseteq\KK[\Mon_S^*(4,12)\setminus
E_7]$. Now, it remains to show that any monomial
 $m\in\Mon_S^*(4,12)\setminus E_7$
can be written as $\frac{g_1\cdots g_7}{x_1\cdots x_7}$ with $g_i\in
V_{3,7}$ and $\det(\Log(g_1,\ldots,g_7))\neq 0$.
 The
case that $|\supp(m)|<7$,  follows by Lemma~\ref{supp<d} and
Proposition \ref{degree6}.  If $|\supp(m)|=7$, we distinguish the
following cases, after a suitable relabeling:

\textbf{Case 1:} For $m=x_1^2\cdots x_7^2$, let $g_1=x_1x_2x_3,
g_2=x_2x_3x_4, g_3=x_3x_4x_5, g_4=x_4x_5x_6, g_5=x_5x_6x_7,
g_6=x_1x_6x_7$ and $g_7=x_1x_2x_7$. Then, the exponent vectors
$\a_1=(1,1,1,0,0,0,0), \a_2=(0,1,1,1,0,0,0), \a_3=(0,0,1,1,1,0,0),
\a_4=(0,0,0,1,1,1,0), \a_5=(0,0,0,0,1,1,1), \a_6=(1,0,0,0,0,1,1)$
and $\a_7=(1,1,0,0,0,0,1)$ are linearly independent and we are done.

\textbf{Case 2:} If $m=x_1^3x_2^2x_3^2x_4^2x_5^2x_6^2x_7$, then
$g_1=x_1x_2x_3, g_2=x_1x_4x_5, g_3=x_1x_5x_6, g_4=x_1x_5x_7,
g_5=x_2x_3x_4, g_6=x_2x_4x_6$ and $g_7=x_3x_6x_7$ will work.

\textbf{Case 3:} If $m=x_1^3x_2^3x_3^2x_4^2x_5^2x_6x_7$, then
$g_1=x_1x_2x_3, g_2=x_1x_3x_4, g_3=x_1x_4x_5, g_4=x_1x_5x_6,
g_5=x_2x_3x_4, g_6=x_2x_5x_7$ and $g_7=x_2x_6x_7$ will work.

\textbf{Case 4:} If $m=x_1^3x_2^3x_3^3x_4^2x_5x_6x_7$, then
$g_1=x_1x_2x_3, g_2=x_1x_3x_4, g_3=x_1x_4x_6, g_4=x_1x_5x_6,
g_5=x_2x_3x_5, g_6=x_2x_3x_7$ and $g_7=x_2x_4x_7$ will work.

\textbf{Case 5:} If $m=x_1^4x_2^2x_3^2x_4^2x_5^2x_6x_7$, then
$g_1=x_1x_2x_3, g_2=x_1x_3x_4, g_3=x_1x_4x_5, g_4=x_1x_5x_6,
g_5=x_1x_6x_7, g_6=x_2x_3x_5$ and $g_7=x_2x_4x_7$ will work.

\textbf{Case 6:} If $m=x_1^4x_2^3x_3^2x_4^2x_5x_6x_7$, then
$g_1=x_1x_2x_3, g_2=x_1x_2x_4, g_3=x_1x_2x_5, g_4=x_1x_3x_6,
g_5=x_1x_4x_7, g_6=x_2x_4x_7, g_7=x_3x_5x_6$.

\textbf{Case 7:} If $m=x_1^4x_2^3x_3^3x_4x_5x_6x_7$, then
$g_1=x_1x_2x_3, g_2=x_1x_3x_4, g_3=x_1x_4x_6, g_4=x_1x_5x_6,
g_5=x_1x_2x_7, g_6=x_2x_3x_5$ and $g_7=x_2x_3x_7$ will work.

\textbf{Case 8:} If $m=x_1^4x_2^4x_3^2x_4x_5x_6x_7$, then
$g_1=x_1x_2x_3, g_2=x_1x_2x_4, g_3=x_2x_3x_5, g_4=x_1x_2x_6,
g_5=x_1x_3x_4, g_6=x_1x_5x_7, g_7=x_2x_6x_7$

\textbf{Case 9:} If $m=x_1^5x_2^2x_3^2x_4^2x_5x_6x_7$, then
$g_1=x_1x_2x_3, g_2=x_1x_2x_4, g_3=x_1x_3x_4, g_4=x_1x_4x_5,
g_5=x_1x_5x_6, g_6=x_1x_6x_7$ and $g_7=x_2x_3x_7$ will work.

\textbf{Case 10:} If $m=x_1^5x_2^3x_3^2x_4x_5x_6x_7$, then
$g_1=x_1x_2x_3, g_2=x_1x_2x_4, g_3=x_1x_2x_5, g_4=x_1x_3x_4,
g_5=x_1x_3x_7, g_6=x_1x_5x_6$ and $g_7=x_2x_6x_7$ will work.

\textbf{Case 11:} If $m=x_1^5x_2^4x_3x_4x_5x_6x_7$, then
$g_1=x_1x_2x_3, g_2=x_1x_2x_4, g_3=x_1x_2x_5, g_4=x_1x_2x_6,
g_5=x_1x_3x_4, g_6=x_1x_5x_7$ and $g_7=x_2x_6x_7$ will work.
\end{proof}

Summing up the above results, we have the following theorem.
\begin{thm} \label{cor}
Let $d\in\{5,6,7\}$. Then
$\GG(\KK[V_{3,d}])=\KK[\Mon^*_S(4,2d)\setminus E_d].$
\end{thm}

 For a monomial ideal $I\subset S$,
 we denote the unique minimal set of monomial generators of  $I$ by $G(I)$. For a monomial $u=x_1^{a_1}\cdots x_n^{a_n}\in S$,
  we set  $\deg_{x_i}(u):=a_i$ for all $i=1, \ldots, n$.
 Let $I$ be generated in a single degree. Then $I$ is said to be a {\em polymatroidal} ideal
 if for any $u,v\in G(I)$ with $\deg_{x_i}(u)> \deg_{x_i}(v)$,
there exists an index $j$ with $\deg_{x_j}(u)< \deg_{x_j}(v)$ such
that $x_j(u/x_i)\in I$. Equivalently,  a monomial ideal $I$ in $S$
is called a
 {\em polymatroidal} ideal, if there exists a set of bases $\mathcal{B}\subset\mathbb{Z}^n$ of a discrete polymatroid, such that
$G(I)=\{{\bf x}^{\bf a} : {\bf a} \in \mathcal{B}\}$. By
\cite[Corollary 6.2]{HH}, we have that the base ring of a discrete
polymatroid is normal and
 Cohen-Macaulay.
The polymatroidal ideals  have been discussed in various articles
with both algebraic and
 combinatorial points of view. An overview of discrete polymatroids
 and polymatroidal ideals is provided in \cite[Chapter 12]{HHbook}.

By \cite[Proposition 3.1]{HJN} the monomial ideal generated by
$\Mon_S(t,r)$ is polymatroidal and  it is easy to see that the
monomial ideal generated by  $\Mon_S^*(t,r)$ is also polymatroidal.

\begin{rem}
Let $A=\KK[V_{2,4}]$ and $S=\KK[x_1,\ldots,x_4]$.
\begin{enumerate}
\item[(a)] By \cite[Theorem 3.2(a)]{HJN},
$\mathbb{G}(A)=\KK[\Mon_S(3,4)\setminus\{x_1x_2x_3x_4\}]$. Let
$u=x_2^2x_3x_4, v=x_1^2x_2x_3\in
\Mon_S(3,4)\setminus\{x_1x_2x_3x_4\}$. Then
$(u/x_2)x_1=x_1x_2x_3x_4\not\in
\Mon_S(3,4)\setminus\{x_1x_2x_3x_4\}$. Hence,  the monomial ideal
generated by $\Mon_S(3,4)\setminus\{x_1x_2x_3x_4\}$ is not
polymatroidal.
\item[(b)] By \cite[Proposition
3.1]{HJN}, \cite[Theorem 3.2(b)]{HJN} and  \cite[Corollary
6.2]{HH}, we have that the monomial ideal generated by
$\Mon_S(3,d)$ is polymatroidal for $d\geq 5$ and $\KK$-algebra
$\mathbb{G}(A)=\KK[\Mon_S(3,d)]$ is normal and Cohen-Macaulay, for
$d\geq 5$.
\end{enumerate}
\end{rem}

 \begin{thm} \label{poly} Let
$d\geq 5$. Then, the monomial ideal generated by
$\Mon_S^*(4,2d)\setminus E_d$ is polymatroidal.
\end{thm}

\begin{proof}
Let $M:=\Mon_S^*(4,2d)\setminus E_d$ and $u=x_1^{a_1}\cdots
x_d^{a_d}, v=x_1^{b_1}\cdots x_d^{b_d}\in M$ and
$\deg_{x_i}(u)>\deg_{x_i}(v)$. Since $\Mon_S^*(4,2d)$ is
polymatroidal, it follows that there exists index $j\in [d]$ such
that $\deg_{x_j}(u)<\deg_{x_j}(v)$ and $(u/x_i)x_j\in
\Mon_S^*(4,2d)$. For convenience, we assume that $i=1$ and $j=2$.

 If $(u/x_1)x_2\not\in E_d$, then there is nothing to prove.
Now, let $(u/x_1)x_2\in E_d$. We consider the following cases:

Case a:  $|\supp(u/x_1)x_2|=4$ and  $|\supp(u)|=4$. Hence,
since $u\not\in E_d$, we have that  $a_1>1$ and  $a_2>0$. We may
assume that $u=x_1^{a_1}x_2^{a_2}x_3^{a_3}x_4^{a_4}$. If for some
$k\in\{5,\ldots,d\}$, $b_k\neq 0$, then $|\supp(u/x_1)x_k|\geq 5$,
and so $(u/x_1)x_k\in M$. Note that $\deg_{x_k}(u)<\deg_{x_k}(v)$.
Now, let $b_5=\cdots=b_d=0$. Since $(u/x_1)x_2\in E_d$, $a_2>0$ and
$u\not\in E_d$, it follows that $a_1=2$.
 Hence, since $|\supp(v)|=4$, $b_5=\cdots=b_d=0$ and
$b_1<a_1=2$, it follows that $b_1=1$, which implies that $v\in E_d$,
a contradiction.

Case b:  $|\supp(u/x_1)x_2|=4$ and $|\supp(u)|=5$. Hence $a_1=1$ and
$a_2>0$. We may assume that $u=x_1x_2^{a_2}x_3^{a_3}x_4^{a_4}
x_5^{a_5}$. If for some $k\in\{6,\ldots,d\}$, $b_k\neq 0$, then
$|\supp(u/x_1)x_k|\geq 5$, and so $(u/x_1)x_k\in M$. Note that
$\deg_{x_k}(u)<\deg_{x_k}(v)$. Now, let $b_6=\cdots=b_d=0$. Since
$b_1<a_1=1$, we have that $b_1=0$.  So $|\supp(v)|=4$. Hence, since
$v\not\in E_d$, we have that $b_i\geq 2$ for $i=2,3,4,5$.  On the
other hand, since $(u/x_1)x_2\in E_d$, it follows that $a_3=1$ or
$a_4=1$ or $a_5=1$. We may assume that $a_3=1$. Now, we claim that
$(u/x_1)x_3\not\in E_d$, and hence $(u/x_1)x_3\in M$.  Assume for
the sake of contradiction, that  $(u/x_1)x_3\in E_d$. Then
$(u/x_1)x_3= x_2^{a_2}x_3^2x_4^{a_4}x_5^{a_5}$, with $a_j=1$ for
some $j\in\{2,4,5\}$ and $a_k+a_l=2d-3$, where
$\{k,l\}=\{2,4,5\}\setminus\{j\}$. Since $a_k,a_l\leq d-2$, we have
$2d-3\leq 2d-4$, a contradiction. So $(u/x_1)x_3\not\in E_d$. Note
that $\deg_{x_3}(v)\geq 2>1=\deg_{x_3}(u).$
\end{proof}

\begin{prop}
Let $d\in\{5,6,7\}$. Then, the $\KK$-algebra
$\mathbb{G}(\KK[V_{3,d}])$ is normal and Cohen-Macaulay.
\end{prop}

\begin{proof}
 Theorem~\ref{cor},  Theorem~\ref{poly} and  \cite[Corollary 6.2]{HH} yield the desired result.
\end{proof}

Based on our evidence, we expect that
$\GG(\KK[V_{3,d}])=\KK[\Mon^*_S(4,2d)\setminus E_d]$, for all $d\geq
5$. For small dimensions  $d\in\{5,6,7\}$, we proved in
Theorem~\ref{cor}.  Lemma~\ref{supp<d}, offers an induction process,
solving the problem for monomials who are not full support. So, the
open challenging part is to check the full support monomials.
However, we leave the higher cases $d\geq 8$, as an open question.
\begin{qu}
Let  $d\geq8$. Is $\GG(\KK[V_{3,d}])=\KK[\Mon^*_S(4,2d)\setminus
E_d]?$
\end{qu}

 If we have positive answer for the above question, then it
follows by Theorem~\ref{poly} and  \cite[Corollary 6.2]{HH}
that $\mathbb{G}(\KK[V_{3,d}])$ is normal and Cohen-Macaulay for
$d\geq 8$.

\bigskip

\paragraph{\bf Data Availability} Data sharing is not applicable to this article as no new data were created
or analyzed in this study.

%%%%%%%%%%%%%%%%%%%%%%%%%%%%%%%%%%%%%%%%%%%%%%%%%%%%%%%%%%%%%%%%%%%%%%%%%%%%%%%%%%%%%%%%%%

\end{document}